\documentclass[10pt,a4paper]{amsart}

\theoremstyle{plain}
\newtheorem{theorem}{Theorem}[section]
\newtheorem{lemma}[theorem]{Lemma}
\newtheorem{proposition}[theorem]{Proposition}
\newtheorem{corollary}[theorem]{Corollary}
\theoremstyle{definition}

\newtheorem{example}[theorem]{Example}

\newtheorem{question}{Question}

\theoremstyle{remark}
\newtheorem{remark}[theorem]{Remark}

\newtheorem{notations}[theorem]{Notations}

\newcommand{\tc}{\ensuremath{\mathcal{T}}}
\newcommand{\oc}{\ensuremath{\mathcal{O}}}
\newcommand{\nc}{\ensuremath{\mathcal{N}}}

\newcommand{\vc}{\ensuremath{\mathcal{V}}}
\newcommand{\ic}{\ensuremath{\mathcal{I}}}

\newcommand{\qc}{\ensuremath{\mathcal{Q}}}

\newcommand{\bX}{\overline{X}}

\newcommand{\ks}{k_{\star}}

\newcommand{\tks}{\tilde{k}_{\star}}

\newcommand{\tY}{\widetilde{Y}}
\newcommand{\tX}{\widetilde{X}}
\newcommand{\tx}{\tilde{x}}
\newcommand{\tH}{\widetilde{H}}
\newcommand{\tD}{\widetilde{\Delta}}

\newcommand{\V}{\mathbb{V}}
\newcommand{\W}{\mathbb{W}}
\newcommand{\Proj}{\mathbb{P}}

\begin{document}

\title[ N\'eron-Severi group of a general hypersurface]
{ N\'eron-Severi group of a general hypersurface}

\author{Vincenzo Di Gennaro }
\address{Universit\`a di Roma \lq\lq Tor Vergata\rq\rq, Dipartimento di Matematica,
Via della Ricerca Scientifica, 00133 Roma, Italy.}
\email{digennar@axp.mat.uniroma2.it}

\author{Davide Franco }
\address{Universit\`a di Napoli
\lq\lq Federico II\rq\rq, Dipartimento di Matematica e
Applicazioni \lq\lq R. Caccioppoli\rq\rq, P.le Tecchio 80, 80125
Napoli, Italy.} \email{davide.franco@unina.it}

\abstract  In this paper we extend the well-known theorem of
Angelo Lopez concerning the Picard group of the general space
projective surface containing a given smooth projective curve, to
the intermediate N\'eron-Severi group of a general hypersurface in
any smooth projective variety.



\endabstract

\maketitle

\section{Introduction}

A well-known result of Angelo Lopez \cite{Lopez}, inspired by a
previous work of Griffiths and Harris \cite{GH}, provides a recipe
for the computation of the N\'eron-Severi group $NS_{1}(S;\mathbb
Z)$ of a general complex surface $S$ of sufficiently large degree
in $\Proj ^3$, containing a given smooth curve. For a smooth
projective variety $X$, we define the $i$-th N\'eron-Severi group
$NS_{i}(X;\mathbb Z)$ as the image of the cycle map $A_{i}(X)\to
H_{2i}(X;\mathbb{Z})\cong H^{2(\dim X-i)}(X;\mathbb{Z})$
(\cite{FultonIT}, $\S 19.1$). This work was intended as an attempt to
extend Lopez's result to the intermediate N\'eron-Severi group
$NS_{\dim X/2}(X;\mathbb Z)$ of a general hypersurface $X$, in any
smooth projective variety. In the previous paper \cite{RCMP} we
already obtained  a generalization, but only in the case of
$\mathbb Q$-coefficients, i.e. only for $NS_{\dim X/2}(X;\mathbb
Q):=NS_{\dim X/2}(X;\mathbb Z)\otimes_{\mathbb Z} \mathbb Q$. More
precisely, in (\cite{RCMP}, Theorem 1.2), we proved the following:

\begin{theorem}
\label{RCMPLopez} Let $Y\subset \Proj =\Proj (\mathbb{C})$ be a
smooth projective variety of dimension $m+1=2r+1$ and set
$\vc_d:={\rm{Im}}(H^0(\Proj,\oc_{\Proj}(d))\to  H^0(Y,\oc_Y(d)))$.
Let $Z\subset Y$ be a closed subscheme of dimension $r$  contained
in a regular sequence of smooth hypersurfaces $\overline{X}\in
\vert\vc_d\vert$, $G_i\in \vert\vc_{d_i}\vert$, $1\leq i \leq r $,
such that $d>d_1 > \dots > d_r$. Let $X\in \vert\vc_d\vert$ be a
very general hypersurface containing $Z$, so that $Z$ is a closed
subscheme of the complete intersection $\Delta:=X\cap G_1\cap\dots
\cap G_r$,
$$\Delta =Z\cup R= (\bigcup_{i=1}^{\rho}Z_i)\cup
(\bigcup_{j=1}^{\sigma}R_j).
$$
Assume that the vanishing cohomology of $X$ is not of pure Hodge
type $(\frac{m}{2},\frac{m}{2})$. Denote by $H^m(X;
\mathbb{Z})_{\Delta}$ the subgroup of $H^m(X; \mathbb{Z})$
generated by the components of $\Delta$, and by $H^m(X;
\mathbb{Z})_{\Delta^-}$ the subgroup of $H^m(X; \mathbb{Z})$
generated by $Z_1, \dots ,Z_{\rho}$, $R_1, \dots ,R_{\sigma-1}$.
Then we have:
\begin{enumerate}
\item $H^m(X; \mathbb{Z})_{\Delta}$ is free of rank $\rho +\sigma $;
\smallskip
\item $NS_r(X;\mathbb{Q})= NS_{r+1}(Y;\mathbb{Q})\oplus H^m(X; \mathbb{Q})_{\Delta^-}$.
\end{enumerate}
\end{theorem}

The aim of this paper is to improve previous Theorem
\ref{RCMPLopez}, showing that:

\begin{theorem}
\label{Lopez}
$$NS_r(X;\mathbb{Z})= \left[NS_r(X;\mathbb{Z})\cap H^m(Y;
\mathbb{Z})\right]\oplus H^m(X; \mathbb{Z})_{\Delta^-}.$$
\end{theorem}

We would like to stress that even though the main troubles in the
proof of Theorem \ref{Lopez} come from the singularities of
$\Delta $, \textit{ such a result is not trivial even for smooth}
$\Delta $. Indeed, although in this case $\tY:=Bl_{\Delta}(Y)$
would be smooth, the strict transform $\tX:= Bl_{\Delta}(X)$ would
vary in a linear system which is   not  very ample on  $\tY$. In
fact, as it is proved in Proposition \ref{contraction}, this
linear system contracts $\bigcap _{i=1}^r G_i$ to a point.
Therefore, one cannot apply Lefschetz Hyperplane Theorem
directly. Actually, it is our opinion that even for smooth $\Delta $ it
would be difficult to avoid the arguments used in this note.

As explained in the body of the paper, the main technical point in
the proof of Theorem \ref{Lopez} refers to the following
Lefschetz-type problem:
\begin{question}
\label{question} {\it Let $G\subseteq \Proj$ be an irreducible,
smooth projective variety of dimension $m=2r\geq 2$, and fix  a
hypersurface $W\in \vert H^0(G, \mathcal O_{G}(d))\vert$ ($d\geq
1$). To what extent one can assume the Gysin map:
\begin{equation}\label{GysinMapIntro}
H_{m+1}(G;\mathbb{Z})\stackrel{\cap \,u}\longrightarrow H_{m-1}(W;\mathbb{Z})
\end{equation}
to be injective (here $u\in H^2(G, G-W;\mathbb{Z})$ denotes
the orientation class \cite{FultonIT}, $\S
 19.2$)?}
\end{question}
Of course the answer to such a question is trivially affirmative
in many cases. If ${\text{Tor}\,}H_{m+1}(G;\mathbb{Z})=0$ or if we
would work with $\mathbb Q$-coefficients then the Gysin map is
injective by Hard Lefschetz  Theorem. If $W$ is smooth then the
Gysin map is injective by Lefschetz Hyperplane Theorem. However,
it is easy to find examples where the above Gysin map is not
injective, see Example \ref{example}. Unfortunately, in our case
$W$ could be singular. The only way to obtain an interesting
result is to  vary $W$. If the linear system $\mid W\mid$ was very
ample outside its base locus, then we could deduce the injectivity
of (\ref{GysinMapIntro}) from Lefschetz Theorem with
Singularities, see (\cite{GM}, p. 199), and compare with
(\cite{RCMP}, Lemma 3.2). Unfortunately, in our case $\mid W\mid$
may not be very ample outside its base locus. This is the ultimate
reason for which the following Theorem, which is the main
technical result of this paper, has required a major effort.

\begin{theorem}
\label{maintech} Keep notations as in Theorem \ref{RCMPLopez}, set
$G:=G_i$, $m =2r:=\dim_{\mathbb C}G$, and define $W:= G\cap X$
($X\in \vert\vc_d\vert$ is Zariski general containing $Z$). Then
the Gysin map
$$\ks: H_{m+1 }(G; \mathbb{Z})\longrightarrow H_{m-1  }(W; \mathbb{Z})
$$
is injective.
\end{theorem}

\begin{remark}
\label{cexample} The following example shows that the condition
$d_i\neq d_j$ in Theorem \ref{RCMPLopez} is necessary. Consider
$Y=\mathbb P^5$. Let $G_1$ be a smooth quadric hypersurface, and
let $L_1$ be a general hyperplane section of $G_1$. Let $G_2$ be a
smooth general quadric hypersurface containing $L_1$, so that
$G_1\cap G_2$ is equal to the union of  $L_1$ with another smooth
quadric threefold $L_2$. Let $X$ be a general hypersurface of
degree $d>2$, and define $\Delta:=Z:=X\cap G_1\cap G_2$. Then
$\Delta$ has two irreducible components $\Delta=Z_1\cup Z_2$, with
$Z_i=X\cap L_i$. Now in $H^4(X; \mathbb{Z})$ we have $Z_1=2H^2$,
where $H$ denotes the hyperplane class. Therefore $H^4(X;
\mathbb{Z})_{\Delta}$ is generated by $H^2$, which contradicts Theorem \ref{RCMPLopez}, (1).
\end{remark}

\bigskip

\section{Some basic facts}
\begin{notations}
(i) From now on, unless it is otherwise stated, all  cohomology
and homology  groups are with $\mathbb{Z}$-coefficients.

\medskip
(ii) {\it Borel-Moore homology}. We will denote by $H_i^{BM}(M)$
the {Borel-Moore homology groups} of a variety $M$. Here we recall
some properties of these groups, which will be needed throughout
the paper. \vskip2mm \noindent a)  Borel-Moore homology is equal
to  ordinary homology for any compact variety (\cite{FultonYT}, p.
217, line 7 from below). \vskip3mm \noindent b) If $U$ is open in
$M$, and $C$ is the complement of $U$ in $M$, then there is a long
exact sequence
\begin{equation}\label{exsq}
\dots \to H_{i+1}^{BM}(U) \to H_i^{BM}(C) \to H_i^{BM}(M) \to
H_i^{BM}(U) \to H_{i-1}^{BM}(C) \to \dots
\end{equation}
(\cite{FultonYT}, Lemma 3, p. 219).
\vskip3mm
\noindent
c) If $M$
is smooth of complex dimension $m$, then there is a natural
isomorphism
\begin{equation}\label{smooth}
H_i^{BM}(M)\cong H^{2m-i}(M)
\end{equation}
(\cite{FultonYT}, (26), p. 217).
\end{notations}

\begin{example}
\label{example} Denote by $T$ an irreducible, projective, smooth
threefold such that ${\text{Tor}\,} H_3(T)\not =0$. Choose a
torsion class $0\not= c\in {\text{Tor}\,}H_3(T)$ and assume  that $l\cdot
c=0$ for some $l\in\mathbb{Z}$ with $l>0$. Define
$$S:=T\times \Proj ^{r-1} \subset G:= T\times \Proj ^{m-3}\subset \Proj, \hskip3mm 2r=m\geq 8,$$
and choose a general $W\in \vert H^0(G,\ic_{S,G}(kl))\vert$, $k\gg
0$. From $\dim S={\text{codim}\,} S+4$ it follows that the hypersurface $W$
gives rise to a section of the normal bundle $\nc_{S,G}(kl)$
which vanishes in dimension four. Therefore, we have $\dim
{\text{Sing}\,}W=4$. Consider the cycle
$$\gamma :=c\otimes [\Proj ^{r-1}]\in H_{m+1}(S),$$
and let $\gamma'$ be the image of $\gamma$ in $H_{m+1}(G)$,
via push-forward. Notice that $\gamma'\neq 0$. From the commutative diagram
$$
\begin{array}{ccccc}
\gamma & \in & H_{m+1}(S) & \longrightarrow & H_{m+1}(G) \\
\downarrow & & \downarrow & & \downarrow \\
\gamma \cap kl[H] & \in &   H_{m-1}(S)& \longrightarrow &  H_{m-1}(W),\\
\end{array}
$$
where $[H]\in H^2(S)$ denotes the hyperplane class, it follows
that the image of $\gamma'$ in $H_{m-1}(W)$ vanishes. Hence the map $H_{m+1}(G)\to H_{m-1}(W)$
provides an example of Gysin map, which is not injective.
\end{example}

\medskip

\begin{remark}
\label{unlikely} As we have just observed, in the examples above
$\dim {\text{Sing}\,}W=4$. We do not know examples of not
injective Gysin maps for hypersurfaces with  isolated
singularities. Keeping
notations as in Theorem \ref{RCMPLopez}, isolated
singularities  appear for instance when  we define $W=G\cap
{\overline{X}}$, $G=G_i$ (\cite{Vogel}, Proposition 4.2.6 and
proof, p. 133). Nevertheless,  even in the case $\dim {\text{Sing}\,}W=0$ it seems
unlikely that Gysin map must be always  injective. Indeed, assume
$\dim {\text{Sing}\,}W=0$ and define
$$\Gamma:= {\text{Sing}\,}W=\{x_1,\dots, x_s\},\hskip2mm W':=W  -\Gamma.$$
Using (\ref{exsq}) and (\ref{smooth}) we have an isomorphism for $m>2$:
$$H_{m-1}(W)\cong H_{m-1}^{BM}(W')\cong
H^{m-1}(W').$$
Consider the
cohomology long exact sequence
$$
\dots\longrightarrow H^{m-1}(W, W') \longrightarrow  H^{m-1}(W)
\longrightarrow H^{m-1}(W') \longrightarrow\dots.
$$
Choose a small ball $S_j\subset G$ around   each $x_j$, and set
$B_j:=S_j\cap W$ and $B_j^0:=B_j-\{x_j\}$. By excision, we have
$$H^{m-1}(W, W')\cong \bigoplus_{j=1}^s H^{m-1}(B_j,B_j^0).
$$
By (\cite{Dimca1}, p. 245), we have
$$  H^{m-1}(B_j,B_j^0)\cong  H^{m-2}(K_j),
$$
where $K_j$ denotes the link of the singularity $x_j$.
By Milnor's Theorem
(\cite{Dimca1}, Theorem 3.2.1, p.76),  the link is
 $(m-3)$-connected. Hence one cannot expect the last
group vanishes. And in fact, when $m=2r$ is even,  for a node and
more generally for an ordinary singularity one has
$H^{m-2}(K_j)\neq 0$. Summing up, we have
\begin{equation}\label{diagram}
\begin{array}{ccccc}
 \bigoplus_{j=1}^s H^{m-2}(K_j) &\longrightarrow &  H^{m-1}(W) & \longrightarrow & H_{m-1}(W) \\
& & \uparrow &\nearrow   &  \\
& & H_{m+1}(G) \cong H^{m-1}(G). &   & \\
\end{array}
\end{equation}
Although the vertical arrow is injective by Lefschetz Hyperplane
Theorem, it seems unlikely  that the oblique one, i.e. the
Gysin map, must be injective   for any $W$. However, we remark
that for certain very special isolated singularities one knows that
$H^{m-2}(K_j)=0$ (\cite{Dimca1}, Proposition 4.7, p.
93, Theorem 4.10, p. 94). Finally, one can infer the injectivity of the Gysin map also when ${\rm{rk}}\, H^{m-2}(W)={\rm{rk}}\, H_{m}(W)$. Indeed, in this case
  the
exact sequence
$$
0\to H^{m-2}(W)\to H_{m}(W)\to\bigoplus_{j=1}^s H^{m-2}(K_j) \to
H^{m-1}(W) \to H_{m-1}(W)
$$
shows that the map $\bigoplus_{j=1}^s H^{m-2}(K_j)\to  H^{m-1}(W)$ is
injective, because  $H^{m-2}(K_j)$ is
torsion free (\cite{Dimca1}, (4.1) and (4.2), p. 91). By (\ref{diagram}), this implies that the Gysin map $H_{m+1
}(G)\to H_{m-1 }(W)$ is injective, because its kernel is a torsion
group by Hard Lefschetz Theorem.
\end{remark}

\begin{notations}\label{LerayHirsch}
Consider a smooth quasi-projective  variety $Y$ of dimension $n$
and a locally free sheaf $\mathcal E$ of rank $r$ on $Y$. Set
$\V:=\Proj ( \mathcal E)$, denote by $\pi: \V\to Y$ the natural projection
and denote by $c:=c_1(\oc _{\V}(1))\in A^1(\V)$ the first Chern class. The
cycle map (\cite{FultonIT}, p.370) sends $A^i(\V )$ into the
Borel-Moore homology group $H^{BM}_{2(n+r-1-i)}(\V)$, which can be
identified with $H^{2i}(\V )$, see (\ref{smooth}). Denote by
$\xi_i\in H^{2i}(\V )$ the cohomology class corresponding to
$c^i\in A^i(\V )$. By the Leray-Hirsch Theorem, we have an
isomorphism for any fixed integer $m$:
$$\phi=\oplus_{i=0}^{r-1}\phi_i:\oplus_{i=0}^{r-1} H^{m-2i}(Y)\to H^m(\V ), \hskip3mm \phi_i(\cdot )=\pi ^*(\cdot )\cup
 \xi_i.$$
\end{notations}

Now we are going to prove that the Leray-Hirsch Theorem holds true also for Borel-Moore homology
groups. The following Lemma is certainly well-known, but we briefly
prove  it for lack of a suitable reference.

\begin{lemma}
\label{LerayHirsch1}
We have an isomorphism of Borel-Moore homology groups:
$$\psi=\oplus_{i=0}^{r-1}\psi_i:H^{BM}_m(\V )\to \oplus_{i=0}^{r-1} H_{m-2i}^{BM}(Y),
\hskip3mm\psi_i(\cdot )=\pi_*(\cdot \cap \xi_i).$$
\end{lemma}
\begin{proof}
As explained in (\cite{Voisin}, Proof of the Leray-Hirsch Theorem, p. 195),
we have an isomorphism in the derived  category $D^*(A_Y)$, notations as in \cite{Dimca2}:
$$\pi_*\mathbb{Z}_{\V }\cong \bigoplus_{i=0}^{r-1}\mathbb{Z}_Y[-2i].
$$
In order to prove the Lemma it suffices to apply the derived
functor $R^{\bullet}\Gamma_c$ to the isomorphism above  and then
take the dual:
$$R^{\bullet}\Gamma_c(\V,\mathbb{Z})\cong  \bigoplus_{i=0}^{r-1}R^{\bullet}\Gamma_c(Y,\mathbb{Z})[-2i],
 \hskip5mm  DR^{\bullet}\Gamma_c(\V,\mathbb{Z})\cong \bigoplus_{i=0}^{r-1}DR^{\bullet}\Gamma_c(Y,\mathbb{Z})[2i].$$
Compare with (\cite{Iversen}, p.374), and use notations as in (\cite{Iversen}, pp. 374-78).
\end{proof}

\medskip

\begin{remark}\label{unisecant}
$(i)$ In the statement of the Leray-Hirsch Theorem the cohomology
classes $\xi_i$ are defined up to classes in $\pi ^* (H^{2i
}(Y))$, hence $\xi_{r-1}$ could be replaced by the cycle class of
any unisecant in $A_n(\V )$.

$(ii)$ Notice that $\pi$ is a local complete
intersection  (l.c.i. for short) morphism \cite{FultonIT}.
Set $M_m:= \ker\left(\oplus_{i=0}^{r-2}\psi_i\right)$. Then
$\psi_{r-1}:M_m\rightarrow H_{m-2r+2}^{BM}(Y)$ is an isomorphism
with inverse  the Gysin map
\begin{equation}
\label{Gysin}
\pi_{\star}: H_{m-2r+2}^{BM}(Y)\to M_{m}\subset H_{m}^{BM}(\V ),
\end{equation}
which represents the tensor product with the fundamental class of
the fiber of $\pi: \V \to Y$. Compare with (\cite{FultonIT},
Example 19.2.1, p. 382), and  with the proof of Theorem 8 in
(\cite{Spanier}, Theorem 8, p. 258).
\end{remark}

\begin{notations}
Choose a section in $H^0(Y,\mathcal E)$, and assume it vanishes on
a subscheme $D\subset Y$ having the right codimension. Then we
have a surjection
$$\mathcal E^{\vee}\longrightarrow  \ic_{D,Y}\longrightarrow 0.$$
This surjection induces an imbedding $\widetilde{Y}:=Bl_{D}(Y)\subset
\V$. Since the natural projection $\pi^Y:= \pi \mid _{\tY}: \tY
\longrightarrow Y$ is a  l.c.i. morphism of codimension $0$, it follows that there exists  a Gysin map
(\cite{FultonIT}, Example 19.2.1, p. 382):
$$\jmath_{\star }:H^{BM}_{\bullet}(Y)\to H^{BM}_{\bullet}(\widetilde{Y}).$$
\end{notations}

\medskip

\begin{theorem}
\label{injective} With notations as above we have: $$\pi^Y_{*}
\circ \jmath_{\star }= \mathrm{id}:H^{BM}_{\bullet}(Y)\to
H^{BM}_{\bullet}({Y}),$$ in particular $\jmath_{\star }$ is
injective.
\end{theorem}
\begin{proof}
We denote by $f:\widetilde{Y}\to \V$ the inclusion morphism. Applying
(\ref{Gysin}) to
 the definition of Gysin map (\cite{FultonIT}, Example 19.2.1, p. 382)  we have:
\begin{equation}\label{defGysin}
\jmath_{\star }(x)= \pi_{\star}(x)\cap u_{\widetilde{Y}},
\hskip3mm \forall x\in H^{BM}_m({Y}).
\end{equation}
Here $u_{\widetilde{Y}}$ denotes the orientation  class of
$\widetilde{Y}$ in $\V$ (\cite{FultonIT},  p. 372), so that:
$$\cap \, u_{\widetilde{Y}}: H^{BM}_{\bullet}(\V)\longrightarrow H^{BM}_{\bullet-2r+2}(\tY).$$
\par\noindent
Since $\widetilde{Y}$ is unisecant in $\V$,  Remark
\ref{unisecant}, $(i)$, implies that we may assume $\xi_{r-1}$ to be the cycle
class of $\widetilde{Y}$.  We thus get
$$ f_*(\cdot \cap \,u_{\widetilde{Y}})=\cdot \cap \xi_{r-1}.
$$
According to (\ref{defGysin}), we have
\begin{equation}\label{Y=xi}
f_*(\jmath_{\star }(x))= f_*(\pi_{\star}(x)\cap
u_{\widetilde{Y}})=\pi_{\star}(x)\cap \xi_{r-1}, \hskip3mm \forall
x\in H^{BM}_m({Y}).
\end{equation}
Using (\ref{Y=xi}), Lemma \ref{LerayHirsch1} and Remark
\ref{unisecant}, $(ii)$, we may conclude
$$(\pi^Y_{*} \circ \jmath_{\star })(x)=\pi_{*}(f_*(\jmath_{\star }(x)))=\pi_*(\pi_{\star}(x)\cap \xi_{r-1})=
\psi_{r-1}\circ \pi_{\star}(x)=x,
$$
for any $ x\in H^{BM}_{\bullet}(Y)$.
\end{proof}

\begin{remark}
Consider a quasi-projective smooth variety $Y$  and a complete
intersection $\Delta = \bigcap_{i=1}^r X_i$, $X_i \in \vert H^0(Y,
\oc_Y(d_i))\vert$. Fix $i_0$, set $X:=X_{i_0}$, and assume that $X$ is
smooth. Applying Theorem \ref{injective} to $Y$ and $\mathcal
E=\oplus_{i=1}^r \oc_Y(d_i)$, and to $X$ and $\mathcal
E=\oplus_{i=1,i\neq i_0}^r \oc_X(d_i)$, we see that the Gysin maps
are injective:
$$\jmath_{\star }:H^{BM}_{\bullet}(Y)\hookrightarrow H^{BM}_{\bullet}(\widetilde{Y}), \hskip2mm \tY:=Bl_{\Delta}(Y),$$
$$\imath_{\star }:H^{BM}_{\bullet}(X)\hookrightarrow H^{BM}_{\bullet}(\widetilde{X}), \hskip2mm \tX:=Bl_{\Delta}(X).$$
Notice that $\tX\subset\tY$ (\cite{FultonIT}, B.6.9, p.436), and
that $\tX$ is a Cartier divisor on $\tY$, for $\Delta $ is
regularly imbedded in both $X$ and $Y$ (\cite{FultonIT}, B.6.10,
p.437).
\end{remark}

\begin{lemma}
\label{commute} Denote by $\iota^X:X\to Y$ and
$\iota^{\widetilde{X}}:\widetilde{X}\to \widetilde{Y}$ the
inclusions. Then the following diagram of Gysin maps is commutative:
$$
\begin{array}{ccc}
 H^{BM}_{\bullet}(Y) & \stackrel{\jmath_{\star }}{\hookrightarrow} & H^{BM}_{\bullet}(\widetilde{Y}) \\
\stackrel{\iota^X_\star}{}\downarrow & & \downarrow \,\stackrel{\iota^{\widetilde{X}}_{\star}}{}\\
H^{BM}_{\bullet -2}(X)& \stackrel{\imath_{\star }}{\hookrightarrow} &  H^{BM}_{\bullet -2}(\widetilde{X}).\\
\end{array}
$$
\end{lemma}
\begin{proof}
The natural maps $\widetilde{X}\stackrel{\pi^X}\longrightarrow X
\stackrel{\iota^X}\longrightarrow Y$ and $
\widetilde{X}\stackrel{\iota^{\widetilde{X}}}\longrightarrow
\widetilde{Y} \stackrel{\pi^Y}\longrightarrow Y$ are equal. Furthermore, they are
l.c.i. maps because they are both composite of l.c.i. maps. Therefore, by
functoriality of the Gysin morphism (\cite{FultonIT}, Example
19.2.1, p. 382), we have:
$$\imath_{\star}\circ
\iota^X_{\star}=\iota^{\widetilde{X}}_{\star}\circ \jmath_{\star}.
$$
\end{proof}

\medskip

\begin{notations}
\label{notationscontraction}  Let $Y\subset \Proj $ be a possibly
singular quasi-projective variety, and set
$\vc_d:={\rm{Im}}(H^0(\Proj,\oc_{\Proj}(d))\to  H^0(Y,\oc_Y(d)))$.
Consider a complete intersection $\Delta = \bigcap_1^r X_i$, $X_i
\in\vert\vc_{d_i}\vert$, with $ d:= d_1 \geq d_{2}\geq  d_{3} \geq
\dots \geq d_r$. Fix a hypersurface
$X\in\vert\vc_{\Delta, d}\vert$, where $\vc_{\Delta, d}:=\vc_d\cap
H^0(Y,\ic_{\Delta,Y}(d))$. Then we have
$$\tX:=Bl_{\Delta}(X)\subset Bl_{\Delta}(Y)=:\tY$$
(\cite{FultonIT}, B.6.9, p.436). Since $\Delta $ is regularly
imbedded in both $X$ and $Y$, it follows   that $\tX$ is a Cartier
divisor on $\tY$ (\cite{FultonIT}, B.6.10, p.437). More precisely
$\tX\in \vert H^0(\oc_{\tY}(d\tH -\widetilde{\Delta}))\vert$,
where $\oc_{\tY}(\tH)$ denotes the pull-back of $\oc_{Y}(1)$ via
the natural projection $\tY\to Y$, and $\widetilde{\Delta}$ denotes the
exceptional divisor in $\tY$. Since $\ic_{\Delta, Y}(d)$ is
globally generated, by letting $X\in \vert\vc_{\Delta, d}\vert$
vary, we have a base point free linear system $\vert \tX \vert$ on
$\tY$ and a morphism
$$\nu : \widetilde{Y}\rightarrow \Proj'= \Proj (\vc_{\Delta, d}^*),\quad \qc:=\nu(\widetilde Y).
$$
\end{notations}

\begin{proposition}
\label{contraction} Assume moreover that $ d > d_{2}$ and set $\tc
:= \bigcap _{i=2}^{r }X_i$. Then we have:
\begin{enumerate}
\item $\tc \cong \widetilde{\tc}:=Bl_{\Delta}(\tc)\subset \widetilde{Y}$, $\widetilde\tc\cap \widetilde{X}= \emptyset$,
hence the morphism $\nu $ sends $\widetilde\tc  $ to a point
$p\in \qc$;
\item the morphism $\nu$ is an isomorphism outside $\widetilde\tc$, namely $\vert\tX\vert$ is very ample on
$\tY - \widetilde\tc$:
$$\nu : \widetilde{Y}-\widetilde\tc \cong \qc - \{p\}.$$
\end{enumerate}
\end{proposition}
\begin{proof}
$(1)$ Since $\Delta $ is a Cartier divisor cut out on   $\tc $ by
$X$, it follows that the natural projection $$\pi:\widetilde{\tc}\to \tc$$ is in
fact an isomorphism.  So we have   $\tc \cong
\widetilde{\tc}=Bl_{\Delta}(\tc)\subset \widetilde{Y}$.
Furthermore, we have:
$$\oc_{\tY}(-\tD)\otimes \oc_{\widetilde{\tc}}\cong \pi^*(\oc_{\tc}(-\Delta))\cong
\pi^*(\ic_{X\cap \tc ,\tc}) \cong \oc_{\tY}(-\tD-\tX)\otimes
\oc_{\widetilde{\tc}}.
$$ Hence we find
$$\oc_{\tY}(\tX)\otimes \oc_{\widetilde{\tc}}\cong
\oc_{\widetilde\tc},
$$
and we are done.
\par\noindent
$(2)$ Consider the point $p\in \Proj'$ representing  the
hyperplane  $L\subset \vert\vc_{\Delta, d}\vert$ spanned by
divisors of the form $X_i\cup M_i$, with $i\geq 2 $  and $M_i\in
\vert \vc_{d-d_i}\vert$. Such a hyperplane is spanned by the image
of $(\vc_{d_2}\cap H^0(Y,\ic_{\tc,Y}(d_2))\otimes \vc _{d-d_2}$ in
$\vert \vc_{\Delta, d}\vert$. Since its base locus is $\tc$, it follows that
$\nu(\widetilde\tc) =p$. On the other hand (\cite{FultonIT},
B.6.10 p.437), we have:
$$N_{\widetilde\tc, \tY}\cong (\pi^*N_{\tc, Y})(-\tD)\cong \oplus _{i=2}^{r-1}\oc_{\tY}(\tX_i).$$
It follows that $\widetilde\tc $ is a complete intersection also
in $\tY$:
$$\bigcap _{i=2}^{r-1}\tX_i=\widetilde\tc \subset\tY.$$
But the hyperplane $L\subset \Proj' \cong \vert\tX\vert^*$  is
spanned by divisors of the form $\tX _i\cup \widetilde{M}_i$, with
$i\geq 2$, $M_i\in  \vert \vc_{d-d_i}\vert$, and
$\widetilde{M}_i:=$ strict transform of $M_i$ in $\widetilde Y$.
Since the base locus of $L$ is $\widetilde\tc$, it follows  that
$\nu^{-1}(p) = \widetilde\tc$ scheme theoretically.
Consider a point $x\in \widetilde{Y}-\widetilde\tc$ and its image
$\nu (x)\not =p$. The corresponding hyperplane $L_x\not= L\subset
\vc_{\Delta, d}$ must contain a divisor $X\in L_x$ such that
$\Delta =X\cap \tc$. If $\nu $ did not separate $x$ from another
point or a tangent vector, then they both would be contained in
$\widetilde{X}:=Bl_{\Delta}(X)$. This is impossible because
$\ic_{\Delta, X}(d_{2})$ is  generated by $\vc_{\Delta, d_2}$,
hence our linear system is very ample on $\widetilde{X}$ (recall
that $d>d_2$).
\end{proof}

\section{Proof of Theorem \ref{maintech}}
\begin{notations}
Let $Y$ be a smooth projective variety of dimension $m=2r+1$, and
let $\overline{X}\in \vert\vc_d\vert$, $G_i\in
\vert\vc_{d_i}\vert$, $1\leq i \leq r $,  be a regular sequence of
smooth hypersurfaces. Assume moreover  that $d>d_1
> \dots > d_r$. Define $\tc:=\bigcap_{i=1}^{r} G_i$ and $\Delta:=\tc \cap \bX$ and
fix $G=G_{i_0}$. If $X\in
\vert\vc_{\Delta,d}\vert$ denotes a general hypersurface
containing $\Delta$, define also   $$W:=X\cap G.$$ Consider the Gysin
map
$$\ks: H_{\bullet }(G)\longrightarrow H_{\bullet-2 }(W),
$$
where $k :W\to G$ denotes the imbedding morphism.
\end{notations}
Theorem \ref{maintech} will follow from a slightly stronger result:

\begin{theorem}
\label{main}
The Gysin map
$$\ks: H_{m+1 }(G)\longrightarrow H_{m-1  }(W)
$$
is injective for a general $W\in \vert\vc_d\cap H^0(G,\ic_{\Delta,
G}(d))\vert$.
\end{theorem}
We start with:

\begin{proposition}
\label{comesfromtc} Assume $r\geq 2$ and define $\tc :=
\bigcap_{i=1}^{r} G_i$. Assume  $x\in H_{m+1}(G)$ is such that
$\ks(x)=0\in H_{m-1}(W)$, for a general $W\in \vert\vc_d\cap
H^0(G,\ic_{\Delta, G}(d))\vert$. Then $x$ belongs to the image of
the push forward    from $\tc$:
$$ x\in {\rm{Im}}(h_*:H_{m+1}(\tc) \to H_{m+1}(G)).
$$
\end{proposition}
\begin{proof}
Denote by $S:={\text{Sing}\,}\Delta$ the singular locus of
$\Delta$, and set
$$
\Delta^0:=\Delta-S,\hskip2mm \tc^0:= \tc- S,\hskip2mm G^0:=G-S,\hskip2mm   W^0:=W-S.
$$
Observe   that $\Delta^0$, $G^0$ and $W^0$ are smooth. Since
$\dim S\leq r-1$ (\cite {IJM}, Proof of Theorem  1.2), it follows that
$H_{m+1}(S)=H_{m}(S)=0$ (\cite{FultonYT}, Lemma 4, p. 219).
Therefore, from the exact sequence for Borel-Moore homology:
$$
\dots \longrightarrow H_{m+1}(S) \longrightarrow H_{m+1}(G)
\longrightarrow H^{BM}_{m+1}(G^0) \longrightarrow
H_{m}(S)\longrightarrow \dots,
$$
we get $H_{m+1}(G)\cong H^{BM}_{m+1}(G^0)$ (compare with
(\ref{exsq})).  We thus find $x\in
H^{BM}_{m+1}(G^0)$, and therefore $\ks (x)=0\in
H^{BM}_{m-1}(W^0)$.
\par
Combining Theorem \ref{injective}  and Lemma \ref{commute}   we
have moreover a commutative diagram with injective horizontal
maps:
$$
\begin{array}{ccccc}
H^{BM}_{m+1}(G^0) &\stackrel{\jmath_{\star}}\hookrightarrow  & H^{BM}_{m+1}(\widetilde{G}^0)  \\
\stackrel {\ks}{}\downarrow &  &\stackrel {\tks }{}\downarrow  \\
H^{BM}_{m-1}(W^0) &\stackrel{\imath_{\star}}\hookrightarrow  & H^{BM}_{m-1}(\widetilde{W}^0), \\
\end{array}
$$
with $\widetilde{G^o}:=Bl_{\Delta^0}(G^0)$ and
$\widetilde{W}^0:=Bl_{\Delta^0}(W^0)$. We thus find
$$\tx:= \jmath_{\star} (x)\in H^{BM}_{m+1}(\widetilde{G^o}), \hskip2mm \textrm{with}
\hskip2mm \tks(\tx)=0\in H^{BM}_{m-1}(\widetilde{W}^0).$$
Let us look at the exact sequence:
\begin{equation}
\label{fromtc} \dots \longrightarrow H^{BM}_{m+1}(\widetilde\tc^0)
\stackrel{\sigma}{\longrightarrow} H^{BM}_{m+1}(\widetilde{G}^0)
\stackrel{\rho}{\longrightarrow}
H^{BM}_{m+1}(\widetilde{G}^0-\widetilde\tc^0) \longrightarrow\dots
\end{equation}
($\tc^0\cong \widetilde{\tc}^0\cong Bl_{\Delta^0}(\tc^0)$).
Applying Notations \ref{notationscontraction} and Proposition
\ref{contraction} to the linear system $\vert \widetilde{W}^0
\vert$ on $\widetilde{G}^0$, we find that $\widetilde{W}^0\cap
\widetilde\tc^0=\emptyset $. Then the linear system   $\widetilde{W}^0$ is very
ample on the smooth variety $\widetilde{G}^0-\widetilde\tc^0$. Since
$$\tks (\tx)=0\in H^{BM}_{m-1}(\widetilde{W}^0)\cong H^{m-1}(\widetilde{W}^0),$$
it follows by Lefschetz Theorem with Singularities (\cite{GM}, p.199) that:
$$\rho(\tx)=0\in H^{BM}_{m+1}(\widetilde{G}^0-\widetilde\tc^0)\cong H^{m-1}(\widetilde{G}^0-\widetilde\tc^0).$$
Then (\ref{fromtc}) implies  $\tx=\sigma(y)\in
{\rm{Im}}(H^{BM}_{m+1}(\widetilde\tc^0) \to
H^{BM}_{m+1}(\widetilde{G}^0))$. We are done because $y\in
H^{BM}_{m+1}(\widetilde\tc^0)\cong H^{BM}_{m+1}(\tc^0)\cong
H_{m+1}(\tc)$, and $h_*(y)\in H_{m+1}(G)$ must coincide with $x$.
In fact they both go to $\tx \in H^{BM}_{m+1}(\widetilde{G}^0)$
(\cite{FultonYT}, p. 219, Exercise 5), and the map
$$H_{m+1}(G)\cong H^{BM}_{m+1}(G^0) \to H^{BM}_{m+1}(\widetilde{G}^0)$$
is injective by Theorem \ref{injective}.
\end{proof}

\begin{proposition}
\label{Tor} Assume $r\geq2$ and define $\tc := \bigcap_{i=1}^{r}
G_i$. If $y\in H_{m+1}(\tc)$ is such that $h_*(y)\in
{\rm{Tor}}\,(H_{m+1}(G))$ then $y=0$.
\end{proposition}
\begin{proof}
First notice that ${\rm{Tor}}\,(H_{m+1}(\tc))=0$. In fact, since
$\dim {\rm{Sing}}\,\tc\leq r-2$ (\cite{IJM}, Proof of Theorem
1.2),  it follows that
$$
H_{m+1}(\tc)\cong H^{BM}_{m+1}(\tc-{\rm{Sing}}\,\tc) \cong
H^{1}(\tc-{\rm{Sing}}\,\tc).
$$
Furthermore, $H^{1}(\tc-{\rm{Sing}}\,\tc)$ is torsion free by the Universal
Coefficient Theorem (\cite{Spanier}, p. 243). From
${\rm{Tor}}\,(H_{m+1}(\tc))=0$ it follows $H_{m+1}(\tc;
\mathbb{Z})\subset H_{m+1}(\tc; \mathbb{Q})$, and we may assume
$y\in H_{m+1}(\tc; \mathbb{Q})$ is such that $0=h_*(x)\in
H_{m+1}(G;\mathbb{Q})$. From now on, in the rest of the proof, all
cohomology and homology groups are with $\mathbb Q$-coefficients.

We are going to argue by induction on $r\geq2 $.

$\bullet\quad r=2$.

In this case, by (\cite{Vogel}, Proposition 4.2.6, p.133), we know
that $\tc=G_1\cap G_2$ is a threefold with isolated singularities
(see also \cite{IJM}, loc. cit.). Set $\Gamma:=
{\rm{Sing}}\,\tc=\{x_1,\dots, x_s\}$, $\tc':=\tc-\Gamma$. Then
$y\in H_{5}(\tc)\cong H^{BM}_{5}(\tc')\cong H^{1}(\tc')$. We claim
that:
\begin{equation}
\label{claim'} y\in {\rm{Im}}(H^{1}(\tc)\to H^{1}(\tc')).
\end{equation}
From the cohomology exact sequence:
$$\dots\longrightarrow H^{1}(\tc) \longrightarrow H^{1}(\tc') \longrightarrow H^{2}(\tc, \tc') \longrightarrow\dots  $$
we see that in order to prove the claim it suffices to show that
$H^{2}(\tc, \tc')=0$. Choose a small ball $S_j\subset G$ around
each $x_j$, and set $B_j:=S_j\cap W$ and $B_j^0:=B_j-\{x_j\}$.
Then by excision we have
$$
H^{2}(\tc, \tc')\cong \bigoplus_{j=1}^s H^{2}(B_j,B_j^0)\cong
\bigoplus_{j=1}^s H^{1}(K_j),
$$
where $K_j$ denotes the link of the singularity $x_j$
(\cite{Dimca1}, p. 245). The claim (\ref{claim'}) follows by
Milnor's Theorem (\cite{Dimca1}, Theorem 3.2.1, p.76). To conclude
the proof in the case $r=2$ it suffices to observe that any $y\in
H^{1}(\tc)\cong H^{1}(G)$ such that $0=h_*(y)\in H_{5}(G)\cong
H^{3}(G)$  vanishes by Hard Lefschetz Theorem. Recall that now we
are assuming that all cohomology and homology groups are with
$\mathbb Q$-coefficients.

\medskip

$\bullet\quad r\geq 3$.

Set $R:=G\cap G_j$, $j\not= i_0$, and denote by $f:\tc \to R$ the
inclusion morphism. We claim that:
\begin{equation}
\label{claim} z:=f_*(y)=0\in H_{m+1}(R).
\end{equation}
 First we have
\begin{equation}
\label{psi0} \psi_*(z)=\psi_*(f_*(y))=(\psi\circ f)_*(y)=h_*(y)=0
\in H_{m+1}(G),
\end{equation}
with $\psi: R\to G$ the inclusion morphism. By (\cite{Vogel},
Proposition 4.2.6, p.133), $R$ has at worst finitely  many
singularities. Set
$$
\Gamma:= {\rm{Sing}}\,R=\{x_1,\dots, x_s\},\hskip2mm R':=R-\Gamma.
$$
Then $z\in H_{m+1}(R)\cong H^{BM}_{m+1}(R')\cong H^{m-3}(R')$.
Consider the cohomology long exact sequence:
\begin{equation}
\label{exactsequence} \dots\longrightarrow H^{m-3}(R)
\longrightarrow H^{m-3}(R') \longrightarrow H^{m-2}(R, R')
\longrightarrow\dots,
\end{equation}
choose a small ball $S_j\subset G$ around  each $x_j$, and set
$B_j:=S_j\cap R$ and $B_j^0:=B_j-\{x_j\}$. By excision we have
\begin{equation}
\label{exactsequence'}H^{m-2}(R, R')\cong \bigoplus_{j=1}^s
H^{m-2}(B_j,B_j^0),
\end{equation}
and by (\cite{Dimca1}, p. 245) we get:
\begin{equation}
\label{exactsequence''}  H^{m-2}(B_j,B_j^0)\cong \bigoplus_{j=1}^s
H^{m-3}(K_j)=0.
\end{equation}
Here $K_j$ denotes the link of the singularity $x_j$. The last
vanishing follows by Milnor's Theorem (\cite{Dimca1}, Theorem
3.2.1, p.76), because the link of an isolated singularity of
dimension $\dim R=m-1$ is $(m-3)$-connected.

Combining (\ref{psi0}),  (\ref{exactsequence}),
(\ref{exactsequence'}) and (\ref{exactsequence''}) we have
$$
z\in H^{m-3}(R)\cong H^{m-3}(G), \hskip2mm 0=\psi _*(z)\in
H_{m+1}(G)\cong H^{m-1}(G),
$$
and our claim (\ref{claim}) follows by Hard Lefschetz Theorem.

Having proved $f_*(y)=0$, we now recall that $\dim
{\rm{Sing}}\,\tc\leq r-2$ by (\cite{IJM}, Proof of Theorem  1.2).
Then we can choose a general hyperplane $H$ and look at the
following commutative diagram:
$$
\begin{array}{ccc}
 H_{m+1}(\tc)\cong H^{1}(\tc') & \stackrel{f_*}{\longrightarrow} & H_{m+1}(R)\cong H^{m-3}(R')\\
\quad \downarrow & & \downarrow \,\,\,\\
H_{m-1}(\tc\cap H)\cong H^{1}(\tc' \cap H) & \longrightarrow &  H_{m-1}(R\cap H)\cong H^{m-3}(R'\cap H),\\
\end{array}
$$
where $\tc':=\tc-{\rm{Sing}}\,\tc$, and the vertical maps are
injective by Lefschetz Theorem with Singularities (\cite{GM}, p.
199). The statement follows  by induction.
\end{proof}

\begin{proof}[Proof of Theorem \ref{main}]
Choose an element $0\not= x\ \in H_{m+1 }(G)$. We have to prove
$0\not= \ks(x)\in  H_{m-1  }(W)$. We distinguish two cases,
according that either $r=1$ or $r\geq 2$.

If $r=1$ then we may assume $x\in H_{3 }(G; \mathbb{Q})$ because
${\rm{Tor}}\,H_{3 }(G)\cong {\rm{Tor}}\,H^{1 }(G)=0$ by the
Universal Coefficient Theorem. And the claim follows because the
composite of $\ks$ with the push-forward (put $m=2$):
$$H^{m-1
}(G;\mathbb{Q})\cong H_{m+1}(G;\mathbb{Q})
\stackrel{\ks}\longrightarrow H_{m-1
}(W;\mathbb{Q})\longrightarrow H_{m-1 }(G;\mathbb{Q})\cong H^{m+1
}(G;\mathbb{Q})$$ is injective by Hard Lefschetz Theorem.

Next assume $r\geq 2$. If $x \notin {\rm{Tor}}\,(H_{m+1 }(G))$
then again we may assume $x\in H_{m+1 }(G;\mathbb{Q})$, and we may
conclude as before. If $0\neq x \in {\rm{Tor}}\,(H_{m+1 }(G))$
then we have  $\ks (x)\neq 0$ just combining Propositions
\ref{comesfromtc} and \ref{Tor}.
\end{proof}

\section{Proof of Theorem \ref{Lopez}}

\begin{notations}
\label{smallcontraction} Applying  Proposition \ref{contraction},
and Notations \ref{notationscontraction}, to the complete
intersection $W=X\cap G$  of Theorem \ref{main},  we get  a
morphism
$$  \widetilde{Y}:=Bl_W(Y)\longrightarrow \qc \subset  \Proj (\vc^*_{W,
d}) \hskip3mm \vc_{W, d}:=\vc_d\cap  H^0(Y,\ic_{W,Y}(d)).$$ This
map  contracts $G\cong \widetilde G:=Bl_W(G) \subset
\widetilde{Y}$ to a point $p\in \qc$, and sends
$\widetilde{Y}-\widetilde G$ isomorphically to $\qc - \{p\}$. By
(\cite{DGF1}, Remark 3.1),  both $\tY $ and $\qc $ have at worst
isolated singularities.
\end{notations}

\begin{corollary}
\label{premain}
The push-forward map:
$$ H_{m+2}(\tY) \longrightarrow H_{m+2}(\qc)$$
is surjective, thus the cokernel of the   map
$$ H_{m+2}(\tY) \longrightarrow H^{m}(X)$$
is torsion free, for a general $X\in \vert\vc_{W, d}\vert$.
\end{corollary}
\begin{proof}
From the commutative diagram
$$
\begin{array}{ccccccc}
H_{k}(\widetilde G)&\stackrel {}{\to}&H_{k}(\tY)&\stackrel{}{\to}
&H_{k}(\tY,\widetilde G)&{\to}&H_{k-1}(\widetilde G)\\
\downarrow  & &\stackrel {}{}\downarrow & & \stackrel {}{}\Vert& &\downarrow  \\
H_{k}(\{p\})&\stackrel
{}{\to}&H_{k}(\qc)&\stackrel{}{\to}&H_{k}(\qc,\{p\})
&{\to}&H_{k-1}(\{p\})\\
\end{array}$$
we see that $H_{m+2}(\tY)\to H_{m+2}(\qc )$ is surjective if the
push-forward $H_{m+1}(\widetilde G)\to H_{m+1}(\tY)$ is injective,
and this follows simply combining Theorem \ref{main} with
Corollary 2.6 of \cite{RCMP}. The last statement is direct
consequence of the first. In fact
$${\rm coker}(H_{m+2}(\tY) \longrightarrow H^{m}(X)) \cong {\rm coker}(H_{m+2}(\qc ) \longrightarrow H^{m}(X)),$$
and the last group is torsion free by Lefschetz Theorem with
Singularities (\cite{GM}, p.199), because
$$H_{m+2}(\qc )\cong H_{m+2}^{BM}(\qc-{\rm{Sing}}\qc )\cong H^{m}(\qc-{\rm Sing}\qc ).$$
\end{proof}

\begin{remark}
\label{explanation} By  Corollary 2.6 of  \cite{RCMP},
$H_{m+2}(\tY)\cong H_{m+2}(Y)\oplus H_m(W)$, hence Corollary
\ref{premain} implies that the group
$${\rm coker}(H_{m+2}(Y)\oplus H_m(W)\to  H^{m}(X))$$
has no torsion. In the morphism above the first component is
intended to be the Gysin map followed by Poincar\'e duality, and
the second one is intended to be the push forward followed by
Poincar\'e duality.
\end{remark}

\begin{notations}
\label{induction} Let $Y\subset \Proj$ be a smooth projective
variety of dimension $m+1=2r+1\geq 3$. Let
$\overline{X},G_1,\dots,G_r$ be a regular sequence of smooth
divisors in $Y$, with $\overline {X}\in \vert\vc_d\vert$, each
$G_i \in \vert\vc_{d_i}\vert$, and such that $d>d_1>\dots>d_r$.
Set $\Delta:=\overline{X}\cap G_1\cap\dots \cap G_r$, and
$W:=\overline X\cap G_1$. For any $1\leq l\leq r-1$ fix general
divisor $H_l\in \vert\vc_{\mu_l}\vert$, with $0\ll
\mu_1\ll\dots\ll \mu_{r-1}$, and for any $0\leq l\leq r-1$ define
$(Y_l,X_l,W_l,\Delta_l)$ as follows. For $l=0$ define
$(Y_0,X_0,W_0,\Delta_0):=(Y,X,W,\Delta)$, $X\in \vert\vc_{W,
d}\vert$ general. For $1\leq l\leq r-1$ define $Y_l:=G_1\cap\dots\cap
G_l\cap H_1\cap\dots\cap H_l$, $X_l:=X\cap Y_l$, $W_l:=X\cap
Y_l\cap G_{l+1}$, and $\Delta_l:=\Delta\cap Y_l$ ($m_l:=\dim
X_l=m-2l$). Notice that $\dim\,Y_{r-1}=3$ and that
$\Delta_{r-1}=W_{r-1}$.
\end{notations}

\begin{remark}
\label{finalremarks}
\begin{enumerate}
\item As in Theorem \ref{Lopez}, define: $$\vc_{l,d}:={\rm{Im}}(H^0(\Proj,\oc_{\Proj}(d))\to  H^0(Y_l,\oc_{Y_l}(d))).$$
\item As in Notations \ref{smallcontraction},  define:
$$ \widetilde{Y}_l:=Bl_{W_l}(Y_l)\longrightarrow \qc_l   \subset\Proj (\vc^*_{W_l, d}),
\hskip3mm \vc_{W_l, d}:=\vc_{l,d} \cap  H^0(Y_l,\ic_{W_l,
Y_l}(d)).
$$
By Corollary \ref{premain} and Remark \ref{explanation} the
group
$$\V_l:={\rm{coker}}(H_{m_l+2}(\qc _l)\to H^{m_l}(X_l))= $$
$$={\rm{coker}}(H_{m_l+2}(Y_l)\oplus H_{m_l}(W_l)\to  H^{m_l}(X_l))$$
is torsion free.
\item  By (\cite{DGF}, Theorem 1.1),  $ \V_l \otimes\, \mathbb{Q}$ supports an
irreducible action of the monodromy group of the linear system
$\vert\vc_{W_l, d}\vert_{\vert_{Y_l}}$. Moreover, by previous
remark, we have  $\V_l \subset \V_l \otimes \mathbb{Q}$.
\end{enumerate}
\end{remark}

Theorem \ref{Lopez} will follow from a slightly stronger result:

\begin{theorem}
\label{Main} Let $Y\subset \Proj$ be a smooth projective variety
of dimension $m+1=2r+1\geq 3$. Let $\overline{X},G_1,\dots,G_r$ be
a regular sequence of smooth divisors in $Y$, with $\overline
{X}\in \vert\vc_d\vert$, each $G_i \in \vert\vc_{d_i}\vert$, and
such that $d>d_1>\dots>d_r$. Set $\Delta:=\overline{X}\cap
G_1\cap\dots \cap G_r$ and let $X\in \vert\vc_d\cap
H^0(Y,\ic_{\Delta,Y}(d))\vert$ be a very general hypersurface
containing $\Delta$. Assume that the vanishing cohomology of $X$
is not of pure Hodge type $(\frac{m}{2},\frac{m}{2})$, denote by
$H^m(X; \mathbb{Z})_{\Delta}$ the subgroup of $H^m(X; \mathbb{Z})$
generated by the components of $\Delta$ and by $H^m(X;
\mathbb{Z})_{\Delta^-}$ the subgroup of $H^m(X; \mathbb{Z})$
generated by the components of $\Delta$ except one. Then we have:
\begin{enumerate}
\item $H^m(X; \mathbb{Z})_{\Delta}$ is freely generated by the components of $\Delta$;
\item $NS_r(X;\mathbb{Z})= \left[NS_r(X;\mathbb{Z})\cap H^m(Y; \mathbb{Z})\right]\oplus H^m(X; \mathbb{Z})_{\Delta^-}$;
\item $NS_r(X;\mathbb{Q})= NS_{r+1}(Y;\mathbb{Q})\oplus H^m(X; \mathbb{Q})_{\Delta^-}$.
\end{enumerate}
\end{theorem}
\begin{proof}[Proof of Theorem \ref{Main}] The argument is very similar to that already used in the proof of
Theorem 3.3 of \cite{RCMP}, so we are going to be rather sketchy.
Thanks to Theorem 1.2 of \cite{RCMP}, it suffices to show that the
cokernel  of the map
$$ H_{m+2}(Y) \oplus H_{m}(\Delta)\longrightarrow H^m(X )
$$
is free. In order to prove this, we argue by decreasing  induction
on $l$ and prove that
$$\W_l :={\rm coker}(  H_{m_l+2}(Y_l)\oplus H_{m_l}(\Delta_l)  \longrightarrow H^{m_l}(X_l ))$$
coincides with the group $\V_l$ defined in   Remark
\ref{finalremarks}, (2). For $l=r-1$ this is clear because
$\Delta_{r-1}=W_{r-1}$, compare with Notations \ref{induction}.
Observe that we only need  to prove the following inclusion:
$${\rm{Im}}(H_{m_l+2}(Y_l) \oplus H_{m_l}(\Delta_l)\longrightarrow H^{m_l}(X_l))
\supseteq {\rm{Im}}(H_{m_l+2}(Y_l)\oplus
H_{m_l}(W_l)\longrightarrow H^{m_l}(X_l))$$ for the reverse
inclusion is obvious. Notice that the composite:
$$H_{m_l}(W_l)\hookrightarrow H_{m_{l+1}}(X_{l+1})\longrightarrow \W_{l+1}$$
vanishes   because $\W_{l+1}= \V_{l+1}$ by induction, the
monodromy acts irreducibly on $\V_{l+1}$ and we can assume that
${\rm{rk}} H_{m_l}(W_l)\ll {\rm{rk}}\V _{l+1}$. Again by induction
we find that
$${\rm{Im}}(H_{m_l}(W_l)\hookrightarrow H_{m_{l+1}}(X_{l+1}))$$ is contained in  $$ {\rm{Im}}(H_{m_{l+1}+2}(Y_{l+1})
\oplus H_{m_{l+1}}(\Delta_{l+1})  \to   H_{m_{l+1}}(X_{l+1 } )),$$
and we are done by a  simple arrow chasing showing that, if
$\alpha \in H_{m_l}(W_l)$ goes to
${\rm{Im}}(H_{m_{l+1}+2}(Y_{l+1})  \to H_{m_{l+1}}(X_{l+1 }))$,
then its push forward in $H_{m_l}(X_l)$ belongs to
${\rm{Im}}(H_{m_{l}+2}(Y_{l})  \to  H_{m_{l}}(X_{l  }))$.
\end{proof}

\end{document}